\documentclass[11pt,titlepage,letterpaper]{article}
\def\standard{\oddsidemargin=0in
              \evensidemargin=0in
              \topmargin =-.6in
              \textheight=9.4in
              \textwidth=6.5in}

{\end{list}}

\def\b1{{\bf 1}}

\def\blot{\quad {$\vcenter{\vbox{\hrule height.4pt
             \hbox{\vrule width.4pt height.9ex \kern.9ex \vrule
width.4pt}
             \hrule height.4pt}}$}}

\standard
\usepackage[]{apacite}
\usepackage{amsthm}
\usepackage{amsmath}
\usepackage{mathrsfs}
\usepackage{dsfont}
\usepackage{psfig}
\usepackage[dvips]{graphics}
\usepackage{setspace}
\usepackage{fullpage}
\usepackage{amsfonts}

\usepackage{color,amssymb,epsfig,shadow,graphicx}
\usepackage{multirow}
\usepackage[top=2.54cm, bottom=2.54cm, left=2.54cm, right=2.54cm]{geometry}
\usepackage{caption}
\usepackage{subcaption}
\usepackage{algorithm}
\usepackage{algpseudocode}
\usepackage{enumitem}
\usepackage{bbm} 
\usepackage{xcolor}
\usepackage{mathtools}

\makeatletter
\def\BState{\State\hskip-\ALG@thistlm}
\makeatother

\usepackage{comment}
\theoremstyle{thmstyletwo}%
\theoremstyle{thmstylethree}%
\newtheorem{definition}{Definition}%
\usepackage[symbol]{footmisc}

\newtheorem{theorem}{Theorem}
\newtheorem{lemma}{Lemma}

\newtheorem{remark}{Remark}

\usepackage[normalem]{ulem}

\theoremstyle{remark}

\theoremstyle{definition}

\begin{document}

\title{Non-linear Multi-objective Optimization with Probabilistic Branch and Bound}

\author{Hao Huang \thanks{Corresponding author.}\\
{\footnotesize Department of Industrial
Engineering and Management,
Yuan Ze University}  \\
{\footnotesize No.135, Yuandong Road, Taoyuan City, Taiwan}\\
{\footnotesize haohuang@saturn.yzu.edu.tw}\\
Zelda B. Zabinsky  \\
{\footnotesize Department of Industrial and Systems Engineering,
University of Washington}  \\
{\footnotesize 3900 E Stevens Way NE, Seattle, WA, USA}\\
{\footnotesize zelda@uw.edu}\\
}

\maketitle

\begin{abstract}
\noindent
A multiple objective simulation optimization algorithm named Multiple Objective Probabilistic Branch and Bound with Single Observation (MOPBnB(so)) is presented for approximating the Pareto optimal set and the associated efficient frontier for stochastic multi-objective optimization problems. MOPBnB(so) evaluates a noisy function exactly once at any solution and uses neighboring solutions to estimate the objective functions, in contrast to a variant that uses multiple replications at a solution to estimate the objective functions.  A finite-time performance analysis for deterministic multi-objective problems provides a bound on the probability that MOPBnB(so) captures the Pareto optimal set.   Asymptotic convergence of MOPBnB(so) on stochastic problems is derived, in that the algorithm captures the Pareto optimal set and the estimations converge to the true objective function values.
Numerical results reveal that the variant with multiple replications is extremely intensive in terms of computational resources compared to MOPBnB(so).  In addition, numerical results show that MOPBnB(so) outperforms a genetic algorithm NSGA-II on test problems.

\noindent\textbf{Keywords:} global optimization; multiple objectives; branch and bound;  stochastic optimization; estimation

\end{abstract}




\maketitle
\section{Introduction}
\label{sec:intro}
Multiple objectives generally exist for practical problems, and providing solutions to multi-objective problems is more challenging than for single objective problems 
\cite{miettinen2012nonlinear}.
A majority of multi-objective studies consider uncertainties in the objective functions \cite{hunter2019introduction}, however, some include uncertainties within decision variables, e.g.,  
\shortciteA{eichfelder2017decision,eichfelder2020algorithmic}.
Simulation is a powerful tool that can easily evaluate multiple performance measures of a system design.
Therefore, simulation optimization algorithms have been developed for optimizing complex multi-objective problems \cite{yoon2020multi}, although most simulation optimization algorithms focus on single objective problems \cite{Fu02,fu2015handbook}.

Many approaches for
multi-objective programming problems assume deterministic linear or convex objective functions with continuous variables, e.g.,  \citeA{Deb01,ehrgott2002multiple}.
More recently, multi-objective programming with linear objective functions and binary/discrete variables are explored, as in \shortciteA{boland2019preprocessing}, \shortciteA{feliot2017bayesian}, \shortciteA{figueira2017easy} and 
\shortciteA{liu2020multi}.
In  simulation optimization problems, objective functions are commonly allowed to be a black-box, which may be non-linear and non-convex. 
Various methodologies have been developed, such as kriging-based algorithms \cite{rojas2020survey} and ranking and selection algorithms \shortcite{li2017optimal}. 

Kriging-based algorithms have primarily been applied to deterministic multi-objective problems. However, recent advancements have extended their application to stochastic multi-objective problems, with a particular emphasis on discretized continuous domains \shortcite{rojas2020survey}. Some of these approaches are also designed to operate in continuous domain spaces \shortcite{horn2017first}. \citeA{rojas2020survey} highlighted that kriging-based algorithms in stochastic settings often integrate various criteria for sampling strategies and noise handling techniques within a general model-based sequential framework. For example, \shortciteA{koch2015efficient} employed a surrogate-assisted Efficient Global Optimization approach, utilizing replication or re-interpolation methods for noise handling. Furthermore, \shortciteA{gonzalez2020multiobjective} advanced this approach by incorporating a ranking and selection method with stochastic kriging to optimize the use of computational resources.

Multi-objective optimal computing budget allocation (MOCBA) has been developed for stochastic multi-objective problems with a limited number of finite designs \shortcite{MOO04OCBA,MOO10OCBA}.
MOCBA derives asymptotic analysis on solution quality, ensuring the probability of correct selections. 
Using the exploration capability of partition-based algorithms, \shortciteA{chew2009differentiated} combines a partition-based algorithm, nested partitions, with MOCBA techniques to extend the capability to problems with more system designs.
\shortciteA{feldman2018score} introduced a method that utilizes rate estimators for bivariate normal objectives. \shortciteA{choi2018pareto} developed a sequential approach to enhance the robustness and manage the complexity inherent in multi-objective ranking and selection methods. 
Additionally, \citeA{andradottir2021pareto} proposed a Pareto set selection procedure that provides guarantees based on the probability of correct selection.
For integer lattices domain space, \citeA{applegate2019r} use a gradient-based method to search the Pareto optimal set.
\shortciteA{li2015mo} propose Multi-Objective Convergent Optimization via Most-Promising-Area Stochastic Search (MO-COMPASS), which is developed for finite and constrained simulation optimization problems.
MO-COMPASS considers Pareto optimality and the asymptotic probability of correct selection, and it has been applied to optimize a hospital inpatient flow process \cite{wang2015multi}.
In addition to these primary approaches used in multi-objective simulation optimization, \shortciteA{kim2011sample} extended the sample average approximation with a trust-region method, originally designed for single-objective simulation optimization, to address multi-objective simulation optimization problems.

Meta-heuristic algorithms for multi-objective deterministic problems with continuous variables include evolutionary algorithms \shortcite{sarker2002new}, genetic algorithms \shortcite{deb2002fast},  particle swarm \cite{yang2011multi}, interacting-particle \cite{MeteZabinsky2014}, and weighted optimization by building deterministic hierarchical bandits \cite{al2019revisiting}.
Within these algorithms, the multiple objective genetic algorithm, NSGA-II \cite{deb2002fast}, is one of the most widely used multi-objective algorithms \shortcite{yusoff2011overview,verma2021comprehensive}.
Although meta-heuristic algorithms are empirically effective, they do not provide statistical analysis to support decision-making, where a common approach in meta-heuristic algorithms is to replicate simulations with different random seeds and average the function values at a design point.

Multi-Objective Probabilistic Branch and Bound with Replications (MOPBnB(wr))\cite{Huang14}, an extension of the single objective version that approximates a level set \cite{Hu13:WinSim,zabinsky2020partition}, is a partition-based algorithm and provides an approach for multi-objective continuous and mixed-integer simulation optimization problems. 
It is different from most partition-based algorithms that provide an asymptotic analysis of convergence, because MOPBnB (wr) and its predecessors provide a finite-time analysis on solution quality and is capable of solving problems with both continuous and integer variables. Whereas the classifications of the subregions in MOPBnB(wr) are permanent, 
\citeA{pedrielli2023part} proposed a single objective level set approximation partition-based algorithm combined with Gaussian process technique, named Part-X, which introduced reclassification of subregions.
\citeA{huang2024multiple} applied a version of MOPBnB (wr) to an emergency department resource allocation problem.

The common use of replication for noise handling in various stochastic multiple objective algorithms may
require a significantly large number of replications for a noisy simulation model to provide an accurate estimation of the full Pareto-optimal set and efficient frontier. A new approach to estimating the function value at a design point with only one replication was introduced in \shortciteA{kiatsupaibul2018single} for continuous variables, and adapted in \shortciteA{FiniteSOSA} for both discrete and continuous variables.  The idea is  to use single observations instead of running multiple replications at a design point, and to estimate the function evaluation at a point by averaging the function values at neighboring designs.  The method slowly shrinks the neighborhood while accumulating more neighboring designs, thus ensuring asymptotic convergence.


Hence, we present an algorithm, called Multi-Objective Probabilistic Branch and Bound with Single Observation (MOPBnB(so)), for Pareto-optimal set approximation allowing mixed continuous and integer variables and noisy function evaluations.  
There are two main differences between MOPBnB(wr) and MOPBnB(so): how they estimate the function value at a point; and the classification of subregions.
MOPBnB(wr) uses multiple replications at a design point and estimates each objective function value with a sample average as in \citeA{Huang14}, whereas MOPBnB(so) follows a more flexible sampling scheme that incorporates  the single observation technique \cite{kiatsupaibul2018single, FiniteSOSA}. 
Secondly, MOPBnB(so) allows reclassification of pruned subregions,  inspired by \citeA{pedrielli2023part}, whereas the pruning in MOPBnB(wr) is permanent.

In Section \ref{sec2}, we describe the details of MOPBnB(so) and highlight the incorporation of the single observation approach for estimation. Section \ref{sec3} discusses the analytical performance of MOPBnB(so), and Section \ref{sec4} demonstrates the numerical performance with test functions.  The conclusion is in Section \ref{sec5}. 

\section{Multiple Objective Probabilistic Branch and Bound}\label{sec2}

The multiple objective simulation optimization problem considered in this study is defined in the following form:
\begin{equation}\label{prob_def}
(\mathcal{P})\text{ } \min_{x\in S}  f_1(x), \ldots,f_m(x)
\end{equation}
where the $m$ objective functions can not be evaluated directly but must be estimated, that is, $f_{\ell}(x)=E_{\Xi}[g_{\ell}(x,\xi_x)], \ell=1,\ldots,m$ and $g_{\ell}(x,\xi_x)$ is a noisy function, typically from a simulation, and $\xi_x$ is the noise term. 
%
%
The domain $S$ has $n$ dimensions, with $n_1$ real-valued variables and $n_2$ integer-valued variables, where $n=n_1+n_2$.  Assume $S$ is closed and bounded, and $S \subset R^{n_1} \times Z^{n_2}$. We assume that a simulation, or oracle call, for a point $x$ can return whether $x$ is feasible (i.e., $x \in S$).

The goal of the algorithm is to approximate the Pareto optimal set,
that is, the set of non-dominated solutions, 
which can be defined as follows.
%
A point $x \in S$ is Pareto optimal if there does not exist another point $y \in S$ such that $f_{\ell}(y) \leq f_{\ell}(x)$ for all $\ell=1, \ldots, m$ and $f_j(y) < f_j(x)$ for at least one $j\in \{1, \ldots, m\}$. We let $x^*$ be a Pareto optimal point.

We present  MOPBnB(so) to approximate the Pareto optimal set.
The details of MOPBnB(so) are shown in the pseudo code in Algorithm 1.



The initialization in \emph{Step 0} sets up user-defined parameters and the initial partition of the solution space. 
MOPBnB(so) inputs the user-defined parameters $\delta$, $\alpha$, $B$, $r_k$, $n_k$, and $c$.
The two parameters $\delta$ and $\alpha$ are used in the performance analysis and are not necessary for computation.  They are included in the pseudo code for clarity.
The parameter $\delta$, $0\leq\delta\leq 1$, represents the closeness of the approximated Pareto set to the true efficient frontier, and $\alpha$,  $0 < \alpha< 1$,  is used to form a statistical quality of solution in terms of $(1-\alpha)$. 
The parameter $B\geq 2$ defines the partitioning scheme for each current subregion at the end of each iteration. 
We keep a list of the current subregions on the $k$th iteration in $\Sigma_k$, and let  $\sigma_i$ denote a subregion. Similarly, we keep track of pruned subregions, while allowing reclassification of subregions.
%
MOPBnB(so)  also needs to initialize three parameters that impact the shrinking of the neighborhood for single observation estimation: the parameter $r_k$ is the radius of the ball defining the neighborhood 
on iteration $k$ and decreases each iteration, i.e., $r_{k+1}<r_k,\ \forall k$; the parameter
$n_k$ represents the number of points to sample in current subregions on iteration $k$; and $c$ is a small number of sample points to take in pruned subregions needed to ensure estimation in the neighborhood.

MOPBnB(so) uniformly samples in subregions in \emph{Step 1}. Besides sampling in the non-pruned current subregions with $n_k$ points, MOPBnB(so) also samples in pruned subregions with a sample size of $kc$ on iteration $k$, so that the number of samples in the neighborhood are large enough to achieve a solution quality. The number of samples in each neighborhood for estimation may overlap with pruned subregions, especially for points close to the boundary of a subregion, so it is possible to sample a small number of points in the pruned subregions and a larger number of samples in the non-pruned current subregions. In the implementation, a small number of samples, $kc$, is used in the pruned subregions, and a suggested value of $\lceil\frac{\ln{\alpha_k}}{\ln{(1-\delta})}\rceil$ is used for $n_k$ in the non-pruned current subregions.


In \emph{Step 2}, MOPBnB(so) introduces the single observation technique for estimating each objective metric.
If the sample is passed from the previous iteration, the function value estimation can be updated from previous estimation as  

\begin{align}
\hat{f}_{\ell}(x_{i,j})&=\frac{\sum_{\{r:x_{i,r}\in B(x_{i,j},r_k)\} }y^\ell_{i,r}}{\vert\{r:x_{i,r}\in B(x_{i,j},r_k)\}\vert} \nonumber
\end{align}

%
\noindent
for $\ell=1, \ldots, m$, where $y^\ell_{i,r}$ is the observed value at $x_{i,r}$, and  $\vert \{ r:x_{i,r} \in B(x_{i,j},r_k) \} \vert$ is the number of samples in the ball around $x_{i,j}$.

In \emph{Step 3}, MOPBnB(so)  uses an existing algorithm in 
\shortciteA{MOKung75onfinding} to identify non-dominated samples based on the estimated performance metrics.

\emph{Step 4} classifies each subregion based on the observation of non-dominated samples. If a subregion contains only dominated samples, the subregion is pruned, whereas subregions with at least one non-dominated sample will be branched into $B$ new subregions.  MOPBnB(so) allows reclassification of pruned subregions since there is a small number of samples implemented in the pruned subregions for single observation estimation purposes.  This may result in reclassifying a subregion to a non-pruned status, with subsequent branching.

\emph{Step 5} checks the stopping condition, which is easily tailored.  The subregions that have not been pruned upon termination provide the approximate Pareto optimal set. In the numerical experiments we stop after a fixed number of iterations.  It is also possible to consider the size of current subregions, and stop when all subregions are either classified as pruned or reach a small enough size to satisfy the user's preferences.





\begin{algorithm}
	\caption{Multi-Objective Probabilistic Branch and Bound with Single Observation (MOPBnB(so))}\label{MOPBnB1}
\begin{algorithmic}
\Procedure{}{}
\BState \emph{Step 0. Initialize: Set user-defined parameters and the initial partition of the solution space.}
\State Input $\delta, \alpha$, $B\geq 2$, $r_{1},\ldots,r_{k},\ldots$,  $n_{1},\ldots,n_{k},\ldots$, and $c$..
	
\State Partition $S$ into $B$ subregions, $\sigma_{1}, \ldots, \sigma_{B}$, 
\State set $\Sigma_1\gets\{\sigma_{1}, \ldots, \sigma_{B} \}$, $\Sigma^P_1\gets\emptyset$, $\alpha_1\gets \alpha/B$, and $k \gets 1$. 

\BState \emph{Step 1. Sample: Uniformly sample additional points in each subregion.} 
\For {each subregion $\sigma_i$ in $\Sigma_k$ and $\Sigma^P_k$, $i=1,\ldots, \vert\Sigma_k\vert+\vert\Sigma^P_k\vert$}

\State uniformly sample additional points
\State such that the total number of points in $\sigma_i$ is
\begin{align}\label{eq_sample2_ball}
	N_{k,i} \gets \left\{\begin{array}{l}
		n_k , i=1, \ldots, \vert\Sigma_k\vert\\
		kc, i=\vert\Sigma_k\vert+1,\ldots, \vert\Sigma_k\vert+\vert\Sigma^P_k\vert
	\end{array} \right.
\end{align}
\EndFor

\State Denote the sample points in $\sigma_i$ by $x_{i,j}$, for $i=1,\ldots, \vert\Sigma_k\vert+\vert\Sigma^P_k\vert$, 
\State and for  $j=1,\ldots, N_{k,i}$.\\
%
\BState \emph{Step 2 Function Value Estimation:}
	\State Evaluate all new samples with a single replication for all objectives as 
	\State $y^l_{i,j}\gets g_{\ell}(x_{i,j},\xi_{x_{i,j}}), \forall i,j,l$.
	\State For every sample, estimate the function value with samples within radius $r_k$.
	\For {each $x_{i,j}\in \sigma_i$, $j=1,\ldots, N_{k,i} $ and $i=1,\ldots, \vert\Sigma_k\vert+\vert\Sigma^P_k\vert$} 
	
	\begin{equation}\label{}
		\hat{f}_{\ell}(x_{i,j})\gets\frac{\sum_{\{r:x_{i,r}\in B(x_{i,j},r_k)\} }y^l_{i,r}}{\lvert\{r:x_{i,r}\in B(x_{i,j},r_k)\}\rvert}, \ell=1, \ldots, m.
	\end{equation}
	\EndFor 
\EndProcedure{\ continued\ next\ page}
\end{algorithmic}
\end{algorithm}

\setcounter{algorithm}{0}
\begin{algorithm}
	\caption{MOPBnB Single Observation (so) continued}\label{MOPBnB3}
	\begin{algorithmic}
	\Procedure{continued}{}
	\BState \emph{Step 3. Identify current non-dominated set: Implement an existing comparison algorithm,  $\mathcal{A}$,
		to find the non-dominated sample points using $\hat{f}_{\ell}(x_{i,j})$ and update $\mathcal{S}_k$ with all non-dominated points and their estimated function values $\hat{f}_{\ell}$.}
	\State $\mathcal{S}_k\gets  \mathcal{A}\left(\{\hat{f}_{\ell}(x_{i,j})\vert \ell=1, \ldots, m,i=1, \ldots, \vert\Sigma_k\vert+\lvert\Sigma^P_k\rvert,j=1, \ldots, N_{k,i}\}\right)$\\
	
	\BState \emph{Step 4. Update: Prune and branch each subregion.}
	\For {each $i$, $i=1,\ldots,\vert\Sigma_{k}\vert+\vert\Sigma^P_{k}\vert$}
	\State Update the pruning indicator  $\emph{P}_i$
	\begin{equation}\label{eqUpP}
		\emph{P}_i\gets\left\{
		\begin{array}{l}
			1, \text{ if  } \sigma_i \cap \mathcal{S}_k = \phi \nonumber.\\
			0, \text{ otherwise} \nonumber 
		\end{array}
		\right.
	\end{equation}
	
	\State set $\alpha_{k+1} \gets \frac{\alpha_{k}}{B}$. 
	
	\If {$\emph{P}_i = 0$ and $\sigma_i$ is branchable} 
	\State partition $\sigma_i$ into $B$ subregions $\bar{\sigma}_{i}^1,\ldots, \bar{\sigma}_{i}^B$ 
	\State update the current set of subregions,
	
	\begin{align}\label{eqUp3}
		&\Sigma_{k+1}\gets\{\bar{\sigma}_i^j : \emph{P}_i = 0, j=1,\ldots,B \} \nonumber
	\end{align}
	
	\ElsIf {$\emph{P}_i = 1$} store the pruned subregions
	\begin{align} 
		&\Sigma^P_{k+1}\gets\{ \Sigma^P_{k}, \bar{\sigma}_i : \emph{P}_i = 1, \} \nonumber
	\end{align}
	\EndIf
	\EndFor
	
	\BState \emph{Step 5. Stopping Condition}:
	\If {all subregions $\sigma_i \in \Sigma_{k+1}$ are small enough to be considered no longer branchable,} terminate the algorithm and output $\Sigma_{k+1}$.
	\Else {  $k\gets k+1$} \textbf{goto} \bf{Step 1}
	\EndIf
	\EndProcedure
\end{algorithmic}
\end{algorithm}

\section{Performance Analysis}\label{sec3}

The performance analysis of MOPBnB consists of two parts.
First, we consider the case when objective functions are deterministic, that is, they can be evaluated with no noise.  In this case, the analysis of MOPBnB(so) follows the analysis of MOPBnB(wr), since no estimation is needed.
Under the no-noise condition, Theorem~\ref{MOPBnBTh0} provides a finite time analysis  on discovering high-quality solutions in the approximation of the Pareto optimal set.

Second, we allow noise in the objective functions, and so an estimation of the objective function is performed.  A finite time performance analysis of MOPBnB(wr) is provided in Theorem~\ref{MOPBnBTh1}, that relies on choosing a
conservative sample size and replication size, iteratively.
The convergence analysis of 
MOPBnB(so),  provided in Theorems~\ref{Theo3} and~\ref{SOMOPBnBTh1}, relies on sampling sufficiently often in shrinking neighborhoods of the single observation technique
to ensure the estimation converges to the true function value for points in the approximation of the Pareto optimal set.

\citeA{Huang14} introduced a performance measure regarding the quality of a Pareto optimal set approximation, which is  considered here for both MOPBnB(wr) and MOPBnB(so) in the performance analysis. 
For each solution $x\in S$, a ``distance" metric $D(x)$ representing the minimum distance  to the true efficient frontier, is used to represent the quality of the solution.
Specifically, it is defined as
\begin{equation}\label{D_def}
	D(x) = \inf_{x^* \in S_E} \left\{  \sqrt{\sum_{\ell=1}^m(f_\ell(x) - f_\ell(x^*))^2} \right\} \end{equation}

	\noindent
	where $S_E$ is the true Pareto optimal set. 
	Although $D(x)$ cannot be evaluated when implementing the algorithm because the true Pareto optimal set is unknown, it  can be used in the performance analysis to reflect the quality of each solution.
	We define a relaxation of the Pareto optimal set by considering high-quality solutions that are in the top  $100\delta\%$ of solutions defined with $D(x)$.  For analysis purposes, we consider  $y(\delta, S)$ and $L(\delta,S)$, where
	$L(\delta,S)$ is the set of high-quality solutions with distance $D(x)$ no more than $y(\delta, S)$.
	Specifically, $y(\delta, S)$ and $L(\delta,S)$ are defined as follows,
	\begin{align}\label{eq0_0_0}
		y(\delta, S) &=\mathop{\arg\min}\limits_{y\in \{D(x):x \in S\}} \{P(D(X)\leq y )\geq \delta \},\text{ for } 0<\delta <1,
	\end{align}
	for $X$ uniformly distributed on $S$, and
	\begin{align}\label{eq0}
		L(\delta, S)=\{ x\in S: D(x)\leq y(\delta, S) \},\text{ for } 0<\delta <1.
	\end{align}
	Theorems \ref{MOPBnBTh0}, \ref{MOPBnBTh1}, and Theorem \ref{SOMOPBnBTh1} use $L(\delta, S)$ to represent the desirable set of high-quality solutions as a relaxation of the Pareto optimal set. To relate the high-quality solutions to the objective functions, we make the following remark.

		%
		%
		\begin{remark}\label{cond1}
			There exists an $\epsilon>0$, such that $\sqrt{m\epsilon^2}=y(\delta,S)$. For all $x$ such that $\vert f_{\ell}(x)-f_{\ell}(x^*)\vert <\epsilon$,  for
			some $x^* \in S_E$, $x\in L(\delta,S)$.
		\end{remark}
		\noindent
		The solution set $\left\{x\in S, {\rm \ with\ } \vert f_{\ell}(x)-f_{\ell}(x^*)\vert <\epsilon \right\}$ is a subset of $L(\delta,S)$, allowing us to develop
		a bound on the objective functions in the performance analysis.

		%
		%

		Theorem \ref{MOPBnBTh0} for deterministic problems shows that the current unpruned subregions for both MOPBnB(so) and MOPBnB(wr) contain some high-quality solutions with probability greater than or equal to $1-\alpha.$  
		In the analysis and in the numerical implementation of MOPBnB(so) on deterministic problems, we set
		\begin{align}\label{eq_sample_deterministic}
			N_{k,i} \gets \left\{\begin{array}{l}
				\bigl\lceil\frac{\ln{\alpha_k}}{\ln{(1-\delta})}\bigr\rceil, i=1, \ldots, \vert\Sigma_k\vert\\
				kc, i=\vert\Sigma_k\vert+1,\ldots, \vert\Sigma_k\vert+\vert\Sigma^P_k\vert
			\end{array} \right.
		\end{align}
		in Step~1.
		The sampling and branching structure, ignoring the performance estimation steps, of MOPBnB(so) is similar to MOPBnB(wr). However, MOPBnB(so) samples in the pruned regions, which increases the probability of sampling in $L(\delta, S)$.
		The conservative bound of Theorem \ref{MOPBnBTh0} still holds for MOPBnB(so).

		\begin{theorem}\label{MOPBnBTh0}
			(c.f. \citeA{Huang14}) Consider the $K^{th}$ iteration of MOPBnB(so) or MOPBnB(wr) on $(\mathcal{P})$ with deterministic objective functions.
			The intersection of the approximated set and $L(\delta,S)$ is non-empty with
			\begin{equation}\label{coro1_0}
				P \left ( \left ( S\setminus \bigcup_{k=1}^K \sigma_p^k\right )\bigcap L(\delta,S) \neq \phi \right )\geq (1-\alpha).
			\end{equation} 
		\end{theorem}
		\begin{proof} 
			See \citeA{Huang14}.
		\end{proof}

		With a proper noise assumption (See Appendix Assumption 1), Theorem~\ref{MOPBnBTh1} is a noisy version of Theorem~\ref{MOPBnBTh0}  for MOPBnB(wr).  Notice the impact of the number of objective functions $m$ on the probability bound.  This is due to the possible mis-ordering of each objective function value.
		
		

		\begin{theorem}\label{MOPBnBTh1}
			(c.f. \citeA{Huang14}) Consider the $K^{th}$ iteration of MOPBnB(wr) on $(\mathcal{P})$ given Assumption 1.
			The intersection of the approximating set and $L(\delta,S)$ is non-empty with
			\begin{equation}\label{coro1_0}
				P \left ( \left ( S\setminus \bigcup_{k=1}^K \sigma_p^k\right )\bigcap L(\delta,S) \neq \phi \right )\geq (1-\alpha)(1-m\alpha).
			\end{equation} 
		\end{theorem}
		\begin{proof} 
			See \citeA{Huang14}.
		\end{proof}
		

		For MOPBnB(so), the performance analysis is derived considering a mixed integer-continuous setting,
		based on the results from \citeA{FiniteSOSA}.
Several assumptions are needed to ensure convergence.
Basically, the random error of the objective function has to be uniformly bounded, and the number of samples  $n_k$ need to increase each iteration to ensure the number of samples in the shrinking ball of radius $r_k$ is  large enough to 
provide a high quality estimation of each objective function with single observation. 
(See the appendix for Assumptions~2 -- 4, repeated  from \citeA{FiniteSOSA}). In Equation (\ref{eq_sample2_ball}), the sample size in the pruned subregions ($kc$) iteratively increases, satisfying Assumption~4.	As  long as the samples in the shrinking ball region increase in a certain manner as specified in the assumptions,  the estimated objective function asymptotically converges to the true objective function,  repeated here in Lemma~\ref{lemma1}.

		
%
%
%
%
		

		\begin{lemma}\label{lemma1}
			(c.f. \citeA{FiniteSOSA}) Given Assumptions 2, 3, and 4, the estimated performance $\hat{f}_{\ell}(x)$ at  $x\in  S$ for any $\ell \in \{1,\ldots,m\}$ converges to the true function $f_{\ell}(x)$ as iterations increase, $\hat{f}_{\ell}(x) \rightarrow f_{\ell}(x)$ with probability one.
		\end{lemma}
		\begin{proof} 
			See  \cite{FiniteSOSA}.
		\end{proof}

		We now derive an asymptotic performance analysis for MOPBnB(so) in Theorems~\ref{Theo3} and~\ref{SOMOPBnBTh1} using single observation for estimation.	
		Theorem \ref{Theo3} states that, if a subregion contains a part of the true Pareto optima set $S_E$, then that subregion will not be pruned, eventually.
		
		
		
		\begin{theorem}\label{Theo3}
			For any subregion that contains a part of the Pareto optima set $S_E$, 
			\begin{equation}\label{Theo2eq_0}
				P\left(\left.\lim_{k\rightarrow\infty} \sigma \in \Sigma_k    \right\vert \sigma \cap S_E\neq \emptyset  \right)=1.
			\end{equation}
		\end{theorem}
		
		\begin{proof} 
			See Appendix.
		\end{proof}

		With confidence that the Pareto optimal set will not be pruned  asymptotically, we next address the relationship 
		the high-quality set $L(\delta,S)$ and the approximation $\Sigma_k$  by MOPBnB(so).
		Theorem \ref{SOMOPBnBTh1} establishes that each non-pruned subregion containing some part of the Pareto optimal set  is a subset of $L(\delta,S)$ asymptotically.
		
		
		
		
		
		\begin{theorem}\label{SOMOPBnBTh1}
			Given the Assumptions 2,3, and 4, for any subregion $\sigma \in \Sigma_k$ and $\sigma\cap S_E\neq\emptyset$, the probability that $\sigma$ is a subset of the high-quality set $L(\delta,S)$ converges to 1 as $k\rightarrow\infty$ , that is
			\begin{equation}\label{SOMOPBnBTh1eq1}
				P\left(\left.\lim_{k\rightarrow\infty} \sigma\subseteq L(\delta,S)\right\vert \sigma \in \Sigma_k, \sigma\cap S_E\neq\emptyset\right)=1.
			\end{equation}
		\end{theorem}
		\begin{proof}
			See Appendix.
		\end{proof}
		
\section{Numerical Experiments}\label{sec4}

In this section, MOPBnB(so)  is numerically compared to MOPBnB(wr) and an extended Non-dominated Sorting Genetic Algorithm II (NSGA-II)  on test functions from the literature. A uniform random search is also considered to provide baseline results. 
NSGA-II is a ``fast sorting and elite multi-objective genetic algorithm" introduced by \citeA{deb2002fast} and widely used for multi-objective problems \cite{yusoff2011overview}. 
The algorithms are executed in MATLAB with  the optimization toolbox.



The algorithms are executed on four test functions:  ZDT1, ZDT2, ZDT3 \cite{Deb01} and the Fonseca and Fleming (FF) function \cite{fonseca1993genetic}.  The efficient frontier for ZDT1 is convex, where ZDT2 and ZDT3 have non-convex and discontinuous efficient frontiers. The
FF function demonstrates a case where the Pareto optimal set is located differently in the domain space.
The test functions are included in the appendix.

These test functions are designed as deterministic functions for multiple objective algorithms. However, we focus on objective functions with noise.
Hence, for the numerical experiments, the objective functions of ZDT1, ZDT2, ZDT3, and FF, $f^0_1(x)$ and $ f^0_2(x)$, are used as deterministic components for a stochastic objective function as follows,
\begin{align}
	&f_{\ell}(x)=E_{\Xi}[g_{\ell}(x,\xi_x)] \nonumber \\
	&g_{\ell}(x,\xi_x) = f^0_{\ell}(x)(1 + \xi_x), \nonumber
\end{align}
where $\xi_x\sim N(0,0.1)$ and ${\ell}=1,2$ in the experiments.


In the numerical experiments, MOPBnB(so) and MOPBnB(wr) use the following parameter values: $\delta=0.1, \alpha=0.1$, and $B = 2$. To initialize replications in MOPBnB(wr), we use $R_1=10$.  Also, to avoid running too many replications, an upper bound on the number of replications for each point is set to be $1,000$.
MOPBnB(so) uses $c=50$. The ball shrinks in the following manner. The ball radius begins with 10\% of the range of test functions' variables and shrinks iteratively incorporating the number of input variables.
The higher the dimensionality $n$ of the test function, the slower the radius shrinks.  The ball radius  on the $k$th iteration is
\begin{equation}
	r_k=\frac{r_0}{B^{k/n}}, \text{\ \  where } r_0=0.1. \nonumber
\end{equation}


Figures \ref{fig:zdt122ite11fMO} and \ref{fig:zdt122ite11xMO} demonstrate a near perfect approximation of the efficient frontier in the objective space and the Pareto optimal set by MOPBnB(so) and MOPBnB(wr) in the two-dimensional domain for FF and ZDT1, respectively. 
In Figure~\ref{fig:zdt122ite11fMO}(a)
and Figure~\ref{fig:zdt122ite11xMO}(a), the
blue stars represent the true efficient frontier and the green and red dots are the estimated  non-dominated solutions provided by MOPBnB(wr) and MOPBnB(so), respectively. 
Figures~\ref{fig:zdt122ite11fMO}(b) and (c) for FF, and Figures~\ref{fig:zdt122ite11xMO}(b) and (c) for ZDT1,
show the approximated Pareto set by MOPBnB(wr) and MOPBnB(so), respectively.  Subregions that are pruned are indicated with boxes.  The number inside each box is the  iteration that the box is  pruned.  The red dots are observed points, and the  blue squares are subregions that approximate the Pareto optimal set.

For the FF test function in Figures~\ref{fig:zdt122ite11fMO}(b) and (c),  the true Pareto set is the line segment $x_1=x_2$ for $x_1, x_2 \in [-1,1]$. 
Similarly, the true Pareto set for ZDT1, in Figures~\ref{fig:zdt122ite11xMO}(b) and (c) is given by $x_2=0$ leading to $g(x)=1$.

In Figure \ref{fig:zdt122ite11fMO}(a),  MOPBnB(wr) shows a relatively more complete efficient frontier for the FF test function than MOPBnB(so).  It can also be observed that MOPBnB(wr) appears to have less bias on the efficient frontier, whereas MOPBnB(so) (the red dots in Figure~\ref{fig:zdt122ite11fMO}(a)) shows minor overestimation of the efficient frontier. However, in the same 12 iterations,  the marginal better results from MOPBnB(wr) are based on 8,811,020 function evaluations including replications in contrast to 6,028 function evaluations for MOPBnB(so). 
In Figure~\ref{fig:zdt122ite11xMO}(a), MOPBnB(so) also appears to discover a smaller fraction of the efficient frontier, however, for the ZDT1 test function, the bias shows a mixture of over and under estimation. 
In this instance, there is also a large computational difference, specifically, 17,641,306 function evaluations for MOPBnB(wr), in contrast to 8,805 function evaluations for MOPBnB(so).



Comparing the approximation of the Pareto optimal set in 
Figures~\ref{fig:zdt122ite11fMO}(b) and (c) for FF, and Figures~\ref{fig:zdt122ite11xMO}(b) and (c) for ZDT1, we see that the proportion of the Pareto optimal set approximated by MOPBnB(so) is smaller than MOPBnB(wr).
However, the scale of the difference on computational resources cannot be ignored; MOPBnB(wr) is more costly to achieve a slightly better result. Hence, in the next set of experiments where we run dimensions of $n=2, 5$, and 10, we
only compare MOPBnB(so) with other algorithms.



\begin{figure}
	\centering
	\begin{subfigure}[b]{0.31\textwidth}
		\centering
		\includegraphics[width=\textwidth]{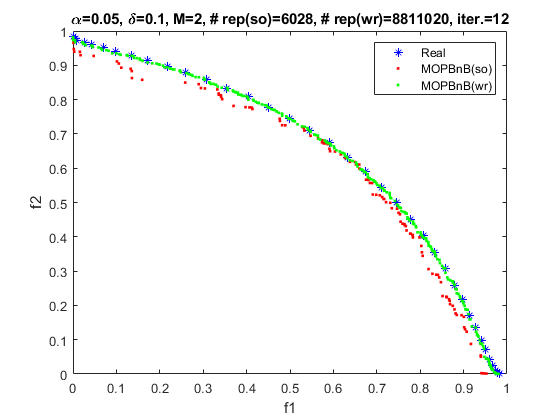}
		\caption{MOPBnB(wr) and (so) \\Efficient frontier}
		\label{fig:a}
	\end{subfigure}
	\hfill
	\begin{subfigure}[b]{0.31\textwidth}
		\centering
		\includegraphics[width=\textwidth]{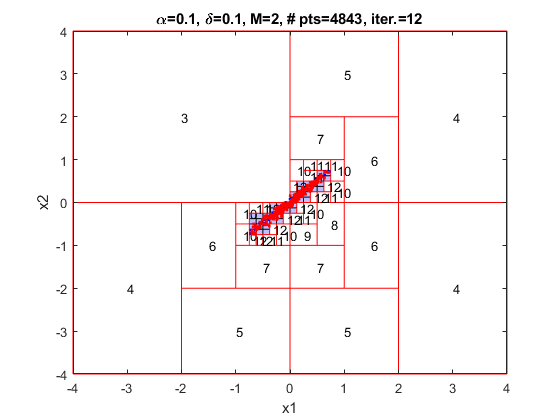}
		\caption{MOPBnB(wr) \\Pareto Optimal Set}
		\label{fig:c}
	\end{subfigure}
	\hfill
	\begin{subfigure}[b]{0.31\textwidth}
		\centering
		\includegraphics[width=\textwidth]{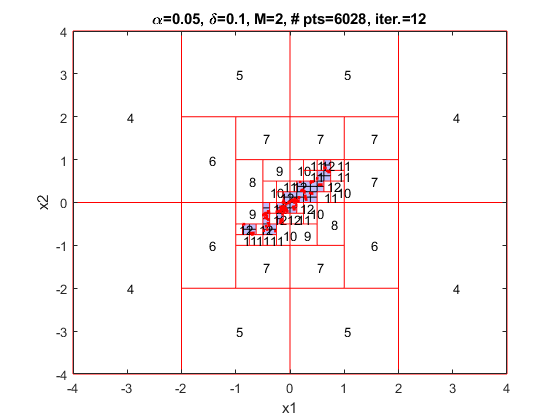}
		\caption{MOPBnB(so) \\Pareto Optimal Set}
		\label{fig:b}
	\end{subfigure}
	\caption{Approximated efficient frontier and Pareto optimal set of FF with $n=2$ using MOPBnB(wr) and MOPBnB(so).}
	\label{fig:zdt122ite11fMO}
\end{figure}

\begin{figure}
	\centering
	\begin{subfigure}[b]{0.31\textwidth}
		\centering
		\includegraphics[width=\textwidth]{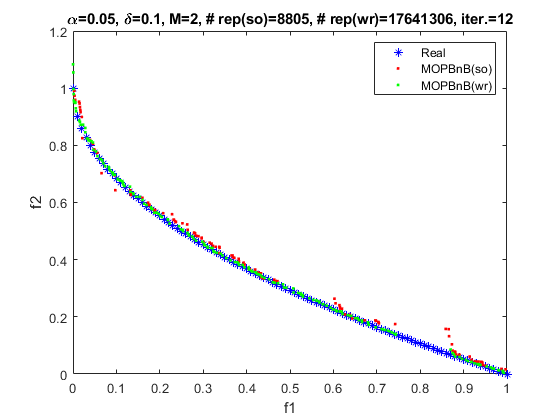}
		\caption{MOPBnB(wr) and (so) \\Efficient frontier}
		\label{fig:a2}
	\end{subfigure}
	\hfill
	\begin{subfigure}[b]{0.31\textwidth}
		\centering
		\includegraphics[width=\textwidth]{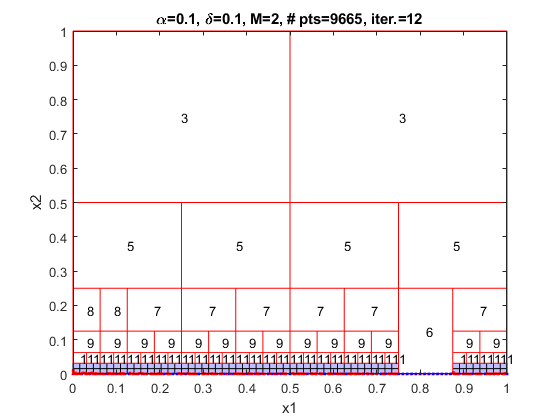}
		\caption{MOPBnB(wr) \\Pareto Optimal Set}
		\label{fig:c2}
	\end{subfigure}
	\hfill
	\begin{subfigure}[b]{0.31\textwidth}
		\centering
		\includegraphics[width=\textwidth]{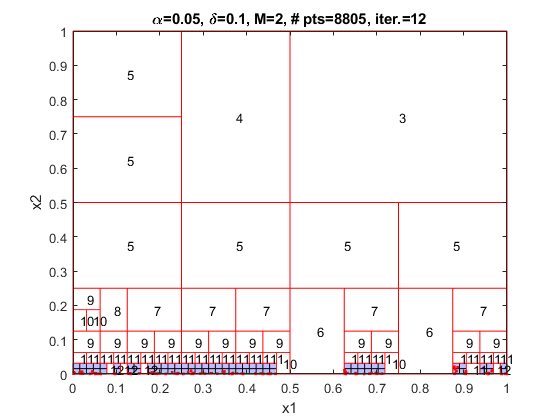}
		\caption{MOPBnB(so) \\Pareto Optimal Set}
		\label{fig:b2}
	\end{subfigure}
	\caption{Approximated efficient frontier and Pareto optimal set of ZDT1 with $n=2$ using MOPBnB(wr) and MOPBnB(so).}
	\label{fig:zdt122ite11xMO}
\end{figure}

For comparison purposes, we execute a uniform random search and the extended NSGA-II, where the function is estimated with a sample average of 20 replications at observed points.  All three algorithms are implemented 50 times, and the performances are averaged over 50 runs.

The numerical comparisons use the three metrics specified in \citeA{Deb01}. 
In Equations (\ref{m1}) through (\ref{m3}), $NS$ is the set of non-dominated points with their function values approximated at each iteration. 
\begin{itemize}
	\item Distance to the true efficient frontier
	\begin{equation}\label{m1}
		M_1(NS)= \frac{1}{\vert NS\vert}\sum_{f'\in NS}\min\{\vert\vert f'-\bar{f}\vert\vert:\bar{f}\in \bar{Y}\},
	\end{equation}	
	where $\bar{Y}$ is the true efficient frontier, approximated through a grid search.
	\item Distribution of the approximated efficient frontier
	\begin{equation}\label{m2}
		M_2(NS)=\frac{1}{\vert NS\vert-1}\sum_{f'\in NS}\left\vert \{q'\in NS:\vert\vert f'-q'\vert\vert>d^*\}\right\vert,
	\end{equation}
	where $d^*$ is a threshold distance to distinguish two solutions, set to $d^*=0.01$ in the experiment.
	\item Extent of the approximated efficient frontier
	\begin{equation}\label{m3}
		M_3(NS)=\sqrt{\max\left\{\vert\vert f'-q'\vert\vert:f',q'\in NS\right\}}.
	\end{equation}
\end{itemize}


Figure \ref{fig:M1zdt123n25} illustrates the comparison between MOPBnB(so), uniform sampling, and NSGA-II using metric $M_1$ on the four test functions, in dimensions $n=2, 5, 10$. 
The line colors distinguish the search algorithms, where blue, green and red are MOPBnB(so), uniform sampling, and NSGA-II, respectively.
The line style indicates the dimensionality $n$ of the test problems.
With the increase of dimensionality of test problems, the trends of the three methods perform in a similar manner.   For the three ZDT test functions, 
MOPBnB(so) decreases the  average distance to  the efficient frontier, $M_1$, the fastest, followed by NSGA-II and uniform sampling.
Different from the ZDT functions, results with FF show MOPBnB(so) has an early lead, but it is taken over by the other two algorithms for $n=2$ and $10$.
After a warm-up period, NSGA-II shows a good performance with respect to $M_1$, but improves slowly which may be the result of a relatively small population size and number of generations.
The  parameter settings for  NSGA-II are taken from default values in MATLAB, with population size set to 50.

\begin{figure}
	\centering
	\includegraphics[width=1\linewidth]{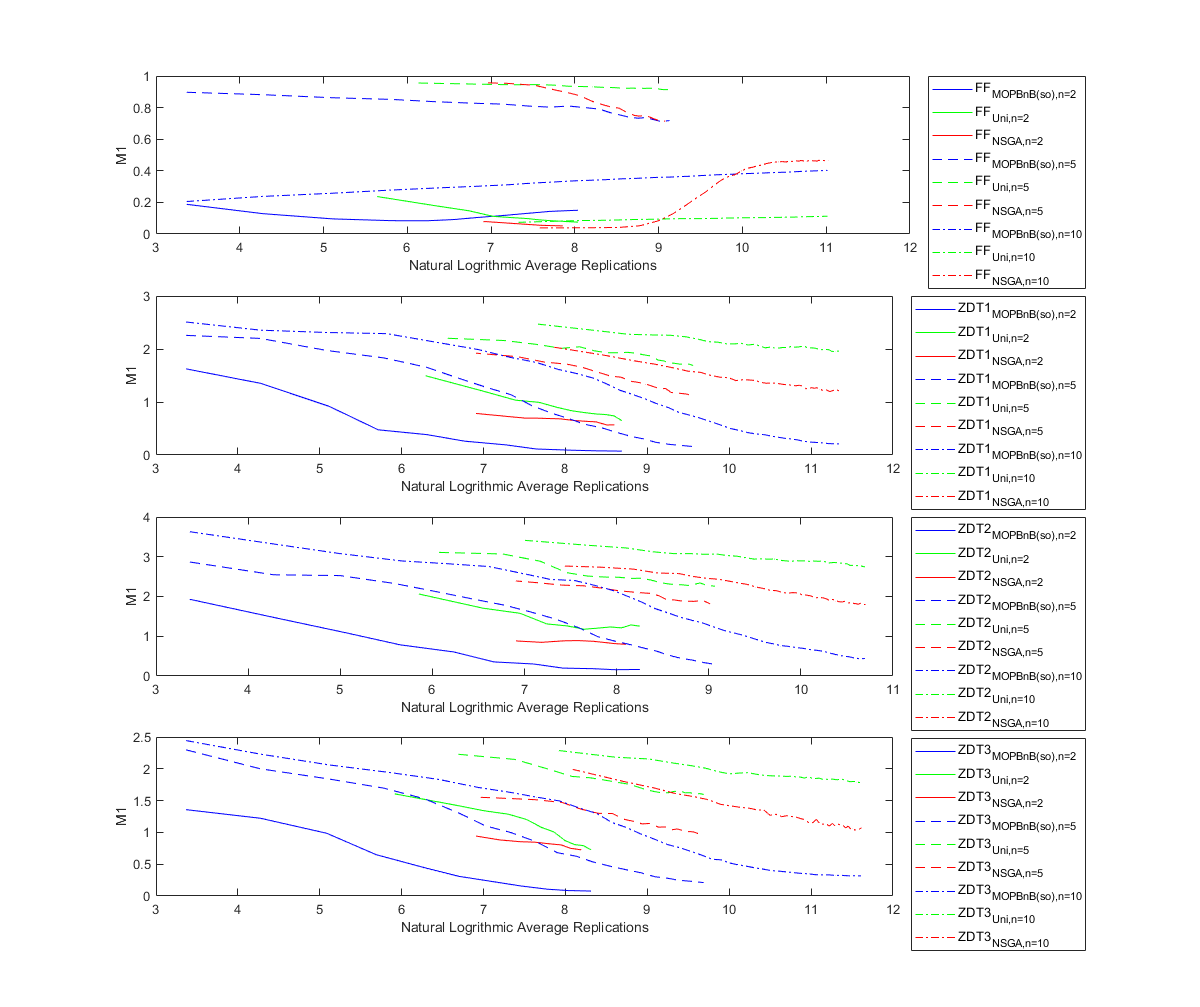}
	\caption[FF, ZDT1, 2, and 3 with n=2, 5, 10]{M1 metric of FF, ZDT1, ZDT2, and ZDT3 with $n=2, 5, 10$ by MOPBnB(so), Uniform Sampling, and NSGA-II.}
	\label{fig:M1zdt123n25}
\end{figure}

Figure \ref{fig:M2zdt123n25} compares MOPBnB(so), uniform sampling, and NSGA-II with the distribution metric $M_2$ on the four test functions with dimensions $n=2, 5, 10$. 
The spread/distribution of the approximated efficient frontier, $M_2$, increases with computational resources.
MOPBnB(so) generally performs better than NSGA-II and uniform sampling, which demonstrates the effectiveness of the single observation technique.
However, for  ZDT2 ($n=5, 10$) and FF ($n=5$), NSGA-II overtakes MOPBnB(so) when the computational resources are large.  
This may be due to the focus MOPBnB(so) puts on subregions with observed non-dominated points, while allocating less computational effort to subregions that are considered for pruning.

\begin{figure}
	\centering
	\includegraphics[width=1\linewidth]{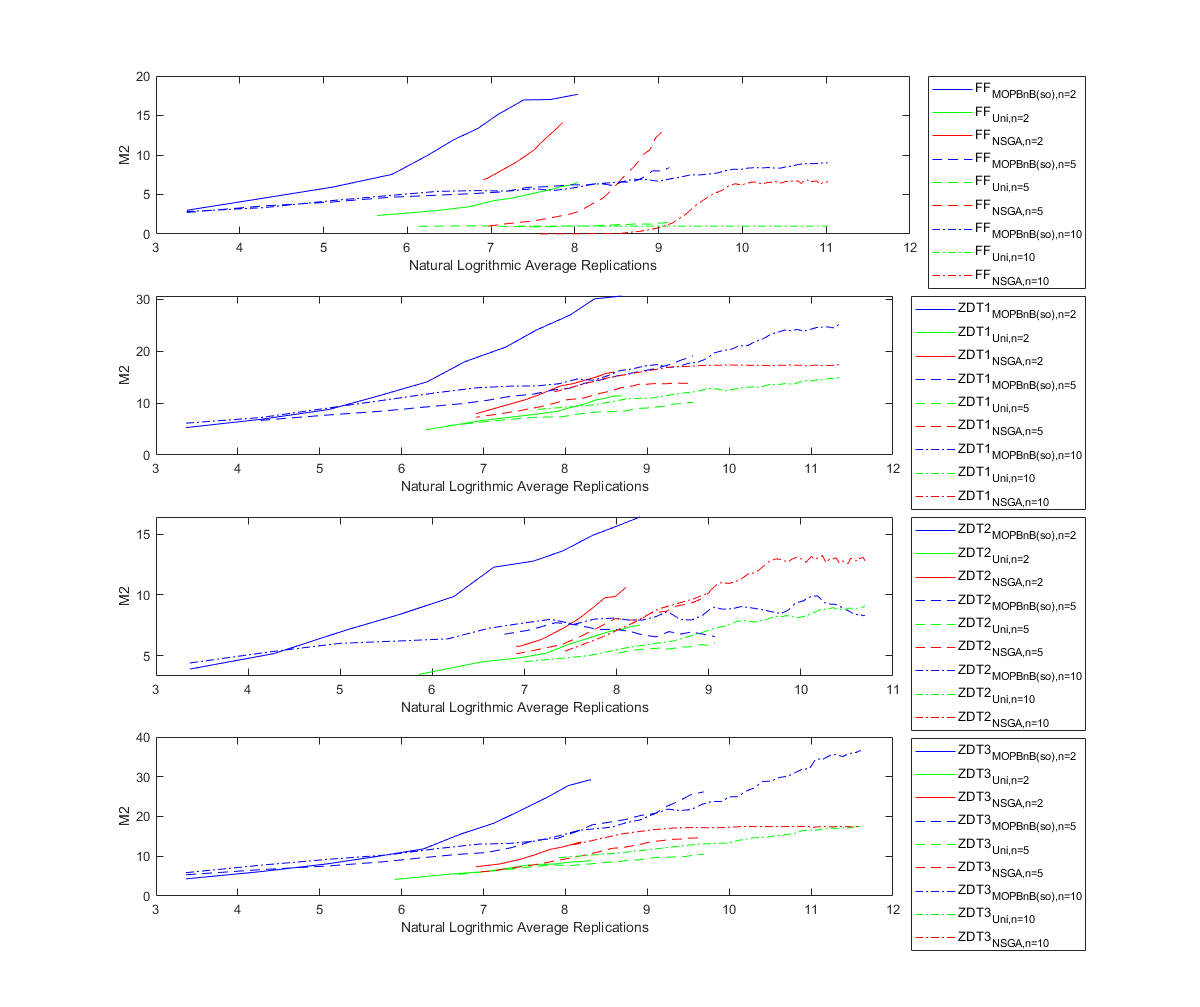}
	\caption[M2 of FF, ZDT1, 2, and 3 with n=2, 5, 10]{M2 metric of FF, ZDT1, ZDT2, and ZDT3 with $n=2, 5, 10$ by MOPBnB(so), Uniform Sampling, and NSGA-II.}
	\label{fig:M2zdt123n25}
\end{figure}

The extent of the approximated efficient frontier, $M_3$, in Figure \ref{fig:M3zdt123n25} shows a different result.
Uniform sampling has the best extent followed by NSGA-II, where MOPBnB(so) loses extent iteratively after the first several iterations, except for FF.  
It is interesting that uniform sampling performs well with respect to this metric, possibly because it has no bias on good regions.  On the other hand, MOPBnB(so) uses its observations to focus the sampling, and parts of  the Pareto optimal set may be pruned, as illustrated  in Figures  \ref{fig:zdt122ite11fMO} and \ref{fig:zdt122ite11xMO}.
In terms of FF, the Pareto optimal set is located  at the center of the solution space and has a relatively short length compared to the width of the solution space.
Therefore, it is relatively difficult for MOPBnB(so) to make incorrect pruning decisions at early iterations.

\begin{figure}
	\centering
	\includegraphics[width=1\linewidth]{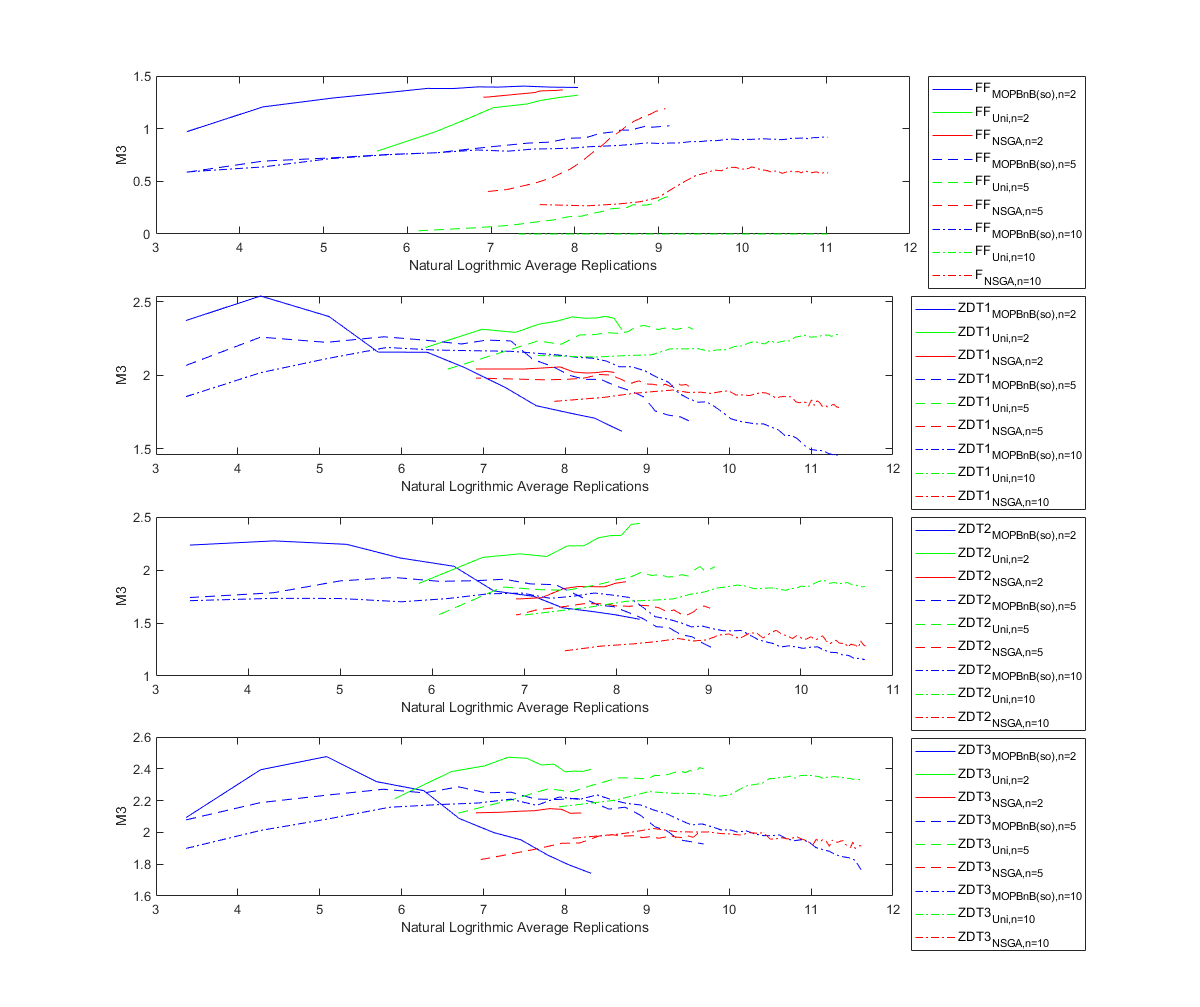}
	\caption[M3 of FF, ZDT1, 2, and 3 with n=2, 5, 10]{M3 metric of FF, ZDT1, ZDT2, and ZDT3 with $n=2, 5, 10$ by MOPBnB(so), Uniform Sampling, and NSGA-II.}
	\label{fig:M3zdt123n25}
\end{figure}

Regardless of the dimensionality of the decision variables, MOPBnB(so) performs better in terms of $M_1$ and $M_2$ under most circumstances, 
however it loses the advantage in terms of $M_3$.
The shape and size of the test functions also affect different performance measures.
In summary, MOPBnB(so) is able to approximate a well spread and close efficient frontier, but some edges of the efficient frontier could be pruned during the process.

\section{Conclusion}\label{sec5}

This study proposes a multi-objective simulation optimization algorithm, MOPBnB(so),
which approximates the Pareto optimal set and its corresponding efficient frontier through iterative branching of the domain space and classifying subregions that are dominated, which leads to focused sampling on subregions thought to be non-dominated.  In contrast to MOPBnB(wr) that uses multiple replications for each sampled solution in order to estimate the objective function value under the stochastic environment, MOPBnB(so)
uses a single observation technique. Hence, MOPBnB(so) can estimate the performance of each sample with only one replication, which significantly reduces the computation. The two variants MOPBnB(so) and MOPBnB(wr) are similar when the objective functions are deterministic, and a
finite-time analysis provides probability bounds to achieve a certain quality of solution.  When the objective functions are estimated with replications, as in MOPBnB(wr), the finite-time analysis still holds because the number of replications are updated iteratively  to statistically ensure the probability bounds are valid. With the performance analysis in  \cite{kiatsupaibul2018single, FiniteSOSA}, MOPBnB(so) can provide asymptotic convergence of the approximated Pareto optimal set to the true Pareto optimal set.
With significantly larger computational resources, MOPBnB(wr) only marginally improves on the performance of MOPBnB(so). 
Furthermore, MOPBnB(so) approximates the Pareto optimal set with non-dominated solutions close to the efficient frontier and well spread out, as compared to NSGA-II and uniform sampling.
However, MOPBnB(so) appears to provide a smaller extent of the approximated efficient frontier as compared to uniform sampling.

It is impossible to implement a simulation optimization problem with infinite computational resources in the real world.
Therefore one of the future directions of this research is to incorporate a finite time analysis on MOPBnB(so), perhaps similar to 
\citeA{al2019revisiting}.

\subsubsection*{Acknowledgements}
This work has been funded in part by the R.O.C.(Taiwan) Ministry of Science and Technology grant MOST 108-2221-E-155-014 and the U.S. National Science Foundation grant CMMI-2204872.

\subsubsection*{Statements and Declarations}
The authors have no financial or non-financial interests to disclose directly or indirectly related to this work. The authors have no competing interests or conflicts of interests for this research.

\clearpage

\bibliography{HFfidelBib3_ZZ}

\section*{Appendix}\label{sec6}

\subsection*{Assumptions for Performance Analysis}

Four assumptions are used in the performance analysis.

The following assumption on the noise of the objective functions is used for Theorem \ref{MOPBnBTh1}.

		\noindent\textbf{Assumption 1.}
Each noisy objective function is normally distributed with an unknown common variance $\sigma^2$, and at each solution $x_j \in S$, the variance can be expressed as $V(x_j) = a_j \sigma^2$ where $a_j$ is a known constant for each $j$, as in \cite{Beck54}.

The following assumptions are necessary for single observation to behave as \citeA{FiniteSOSA} proposed.

\noindent\textbf{Assumption 2.} All objective functions, $f_1(x), \ldots,f_m(x)$, are continuous on the convex part of $S$.

\noindent\textbf{Assumption 3.} The random error $(g_{\ell}(x,U)-f_{\ell}(x))$ is uniformly bounded over $x \in S$; that is, there exists $0 < \alpha < \infty$ such that, for all $x \in S$ with probability one, $\vert g_{\ell}(x, U)-f_{\ell}(x)\vert <\alpha$.


\noindent\textbf{Assumption 4.}
The sequence of radii $r_k$ and sample numbers $n_k$ are such that the number of sample points in the balls grows faster than the radii
shrink, see Assumption 4 in \citeA{FiniteSOSA}.

\subsection*{Proof of Theorem 3}

\begin{proof} 
	Given $\Sigma_k\cup\Sigma^P_k=S$, (\ref{Theo2eq_0}) is equal to 
	\begin{equation}\label{Theo2eq_1}
		P\left(\left. \sigma \in \Sigma^P_k  \text{ } i.o.  \right\vert\sigma \cap S_E\neq \emptyset\right)=0
	\end{equation}
	where $i.o.$ stands for infinitely often.
	Considering the most difficult case that only one point $x$ in the pruned subregion $\sigma$ is a part of the Pareto optima set, 
	(\ref{Theo2eq_1}) can only be violated when $x$ is not sampled and the subregion is pruned incorrectly. 
	
	Let   $x'$ be a sampled point in $\sigma$ that is closest to $x$ in the domain.  By Assumption 2, there exists a finite bound $\rho$ on the differences between objective function values, i.e.,
	%
	\begin{equation}\label{fbound1}
		\vert f_{\ell}(x)-f_{\ell}(x')\vert<\rho, \ 
		\forall \ell. 
	\end{equation}
	Since $\sigma$ has been pruned, there exists a sampled point, call it 
	$x^*$, in the sampled non-dominated set $\mathcal{S}_k$ that dominates $x'$ given the estimated values, that is,
	$\hat{f}_{\ell}(x')\geq\hat{f}_{\ell}(x^*)$ and there exists $j$ such that  $\hat{f}_j(x') > \hat{f}_j(x^*)$. 
	The idea is to show that 
	\begin{align}\label{Theo2eq0}
		&P\left(\left. \hat{f}_{j}(x')-\rho>\hat{f}_{j}(x^*)  \text{ } i.o.\right\vert x'\in\sigma, x^*\in \mathcal{S}_k, \sigma\in\Sigma^P_k \right)=0. 
	\end{align}

	
	%
	%

	According to Lemma 1, as $k\rightarrow\infty$, $\vert \hat{f}_{\ell}(x)-f_{\ell}(x)\vert\rightarrow 0$ for all $x\in S$.
	Given the sampling scheme from MOPBnB(so), each iteration has $kc$ more samples selected from the pruned subregions.
	Therefore, the distance between $x$ and $x'$ decreases as $k$ increases.
	The combination of Lemma \ref{lemma1} for $x'$ and $x^*$ and the closeness in (\ref{fbound1}) can be written as
	\begin{equation}\label{The1eq3}
		P\left\{\lim_{k\rightarrow\infty}\left(\hat{f}_\ell(x')=f_\ell(x')\right)\cap\left(\hat{f}_\ell(x^*)=f_\ell(x^*)\right)\cap\left(\rho=0\right)\right\}=1.
	\end{equation}
	Let $z_0=\min\{z\in \mathbb{Z}^+\colon (1/z) < \vert f_\ell(x') - f_\ell(x^*) \vert+ \rho\}$ be  a bound on the estimation error that could happen. Equation (\ref{The1eq3})  also implies,
	
	\begin{align}\label{The1eq4}
		P&\left\{\bigcup_{z=z_0}^{\infty}\left(\left(\vert \hat{f}_\ell(x')-f_\ell(x')\vert>\frac{1}{3z}\right)\bigcap\left(\vert\hat{f}_\ell(x^*)-f_\ell(x^*)\vert>\frac{1}{3z}\right)\bigcap\left(\rho>\frac{1}{3z}\right)\text{ }i.o.\right)\right.\nonumber\\
		&\hspace{0.5in} \left\vert \phantom{\bigcup_{z=z_0}^{\infty}} \hspace{-0.2in} x'\in\sigma, x^*\in \mathcal{S}_k, \sigma\in\Sigma^P_k  \right\}=0.
	\end{align}
	
	\noindent	For any $z>z_0$, we consider a nesting of the following events,
	\begin{align}\label{The1eq5}
		&\left\{\left.\left(\left(\vert \hat{f}_\ell(x')-f_\ell(x')\vert>\frac{1}{3z}\right)\bigcap\left(\vert\hat{f}_\ell(x^*)-f_\ell(x^*)\vert>\frac{1}{3z}\right)\bigcap\left(\rho>\frac{1}{3z}\right)\text{ }i.o.\right)\right\vert x'\in\sigma, x^*\in \mathcal{S}_k, \sigma\in\Sigma^P_k \right\}\nonumber\\
		&\supseteq \left\{\left.\left(\left(\hat{f}_\ell(x')<f_\ell(x')-\frac{1}{3z}\right)\bigcap\left(\hat{f}_\ell(x^*)>f_\ell(x^*)+\frac{1}{3z}\right)\bigcap\left(\rho>\frac{1}{3z}\right)\text{  }i.o.\right)\right\vert x'\in\sigma, x^*\in \mathcal{S}_k, \sigma\in\Sigma^P_k\right\}\nonumber\\
		&\supseteq \left\{\left.\left(\hat{f}_{\ell}(x') > \hat{f}_{\ell}(x^*)\right)\cap\left(\rho>0\right)\text{ }i.o.  \right\vert x'\in\sigma, x^*\in \mathcal{S}_k, \sigma\in\Sigma^P_k\right\}.
	\end{align}

	Combining (\ref{The1eq4}) and (\ref{The1eq5}) with (\ref{Theo2eq0}), we have
	\begin{align*}\label{The1eq6}
		&P\left(\left. \hat{f}_{\ell}(x')-\rho>\hat{f}_{\ell}(x^*)  \text{ } i.o.\right\vert x'\in\sigma, x^*\in \mathcal{S}_k, \sigma\in\Sigma^P_k \right)\nonumber\\
		&\leq P\left\{\left.\left(\left(\vert \hat{f}_\ell(x')-f_\ell(x')\vert>\frac{1}{3z}\right)\bigcap\left(\vert\hat{f}_\ell(x^*)-f_\ell(x^*)\vert>\frac{1}{3z}\right)\bigcap\left(\rho>\frac{1}{3z}\right)\text{ }i.o.\right)\right\vert x'\in\sigma, x^*\in \mathcal{S}_k, \sigma\in\Sigma^P_k \right\}\nonumber\\
		&= 0.
	\end{align*}
	Therefore, $P\left(\left.\lim_{k\rightarrow\infty} \sigma \in \Sigma_k    \right\vert \sigma \cap S_E\neq \emptyset  \right)=1.$
\end{proof}

\subsection*{Proof of Theorem 4}

\begin{proof}
	In order to prove Equation (\ref{SOMOPBnBTh1eq1}), 
	we show that
	\begin{equation}\label{SOMOPBnBTh1eq2}
		P\left(\left. \sigma\nsubseteq L(\delta,S)\text{ } i.o.\right\vert \sigma \in \Sigma_k, \sigma\cap S_E\neq\emptyset\right)=0.
	\end{equation}
	If $\sigma \nsubseteq  L(\delta, S)$, it means at least one solution $x$ in $\sigma$ is not part of $L(\delta, S)$, 
	hence, 
	\begin{equation*}\label{SOMOPBnBTh1eq2}
		P\left(\left. \sigma\nsubseteq L(\delta,S)\text{ } i.o.\right\vert \sigma \in \Sigma_k, \sigma\cap S_E\neq\emptyset\right)=P\left(\left. 
		\exists x\in\sigma {\rm \ and\ } x \notin L(\delta,S)
		\text{ } i.o.\right\vert \sigma \in \Sigma_k, \sigma\cap S_E\neq\emptyset\right).
	\end{equation*}
	
	According to Remark \ref{cond1}, 
	the set $\left\{x\in S, {\rm \ with\ } \vert f_{\ell}(x)-f_{\ell}(x^*)\vert <\epsilon \right\}$ is a subset of $L(\delta,S)$.
	Therefore, if $x \notin L(\delta,S)$, it is also not in that set, so at least one objective function exceeds the threshold, $\vert f_{\ell}(x)-f_{\ell}(x')\vert\geq\epsilon$, where $x'\in \sigma\cap S_E$. We have,
	
	\begin{align}\label{SOMOPBnBTh1eq3}
		&P\left(\left. 
		\exists x\in\sigma {\rm \ and\ } x \notin L(\delta,S)
		\text{ } i.o.\right\vert \sigma \in \Sigma_k, \sigma\cap S_E\neq\emptyset\right) \nonumber \\
		&= 	P\left(\bigcup_{\ell=1}^m \left(\vert f_{\ell}(x)-f_{\ell}(x')\vert\geq\epsilon\right) \text{ } i.o. \vert\exists x, x' \in \sigma \in\Sigma_{k}, x'\in S_E \right). 
	\end{align}

	Since MOPBnB(so) branches non-pruned subregions iteratively, eventually the subregion will be small enough that we achieve a bound on the objective functions, as in \eqref{fbound1}, where $\rho \leq \epsilon$.  Also, since there exists a Pareto solution $x'\in\sigma$, all solutions within $\sigma$ have objective function differences less than $\epsilon$, that is $\vert f_{\ell}(x)-f_{\ell}(x')\vert<\epsilon$, $\forall\ell$.  Therefore, we have, 
	
	\begin{align}\label{SOMOPBnBTh1eq4}
		&P\left(\bigcup_{\ell=1}^m \left(\vert f_{\ell}(x)-f_{\ell}(x')\vert\geq\epsilon\right) \text{ } i.o. \vert\exists x, x' \in \sigma \in\Sigma_{k}, x'\in S_E \right) = 0, 
	\end{align}
	which indicates that 
	$P\left(\left.\lim_{k\rightarrow\infty} \sigma\subseteq L(\delta,S)\right\vert \sigma \in \Sigma_k, \sigma\cap S_E\neq\emptyset\right)=1.$
\end{proof}

\subsection*{Test Functions}

The algorithms are executed on four test functions:  ZDT1, ZDT2, ZDT3 \cite{Deb01} and the Fonseca and Fleming (FF) function \cite{fonseca1993genetic}.  All three ZDT functions have the same general formulation, with two objective functions and $n$-dimensional, real-valued decision variables $x=(x_1,\ldots,x_n)$,
\begin{align}
	&\min\  f^0_1(x_1), f^0_2(x) \nonumber \\
	&\text{where\ } f^0_2(x)=g(x_2,\ldots,x_n)\ h(f_1^{ 0}(x_1),g(x_2,\ldots,x_n))\nonumber
\end{align}

Specifically, each function is defined as follows.
\begin{itemize}
	\item ZDT1 has a convex efficient frontier:
	\begin{align}
		&f^0_1(x_1)=x_1\nonumber\\
		&g(x_2,\ldots,x_n)=1+9\sum_{i=2}^{n}\frac{x_i}{n-1} \nonumber \\
		&h(f_1^{0}(x_1),g(x_2,\ldots,x_n))=1-\sqrt{f_1^{ 0}(x_1)/g(x_2,\ldots,x_n)},\nonumber
	\end{align}
	
	where 
	$x_i\in[0,1]$, $i=1,\ldots,n$, and $n=2,\ldots,30$. The efficient frontier is formed with $g(x_2,\ldots,x_n)=1$.
	\item ZDT2 has a non-convex efficient frontier:
	\begin{align}
		&f^0_1(x_1)=x_1\nonumber\\
		&g(x_2,\ldots,x_n)=1+9\sum_{i=2}^{n}\frac{x_i}{n-1} \nonumber \\
		&h(f_1^{0}(x_1),g(x_2,\ldots,x_n))=1-\left( f_1^{0}(x_1)/g(x_2,\ldots,x_n) \right)^2,\nonumber
	\end{align}
	where $x_i\in[0,1]$, $i=1,\ldots,n$, and $n=2,\ldots,30$. The efficient frontier is formed with $g(x_2,\ldots,x_n)=1$.
	\item ZDT3 has a  discontinuous efficient frontier consisting of several convex parts:
	\begin{align}
		&f^0_1(x_1)=x_1\nonumber\\
		&g(x_2,\ldots,x_n)=1+9\sum_{i=2}^{n}\frac{x_i}{n-1} \nonumber \\
		&h(f_1^{0}(x_1),g(x_2,\ldots,x_n))=1-\sqrt{f_1^{0}(x_1)/g(x_2,\ldots,x_n)} \nonumber \\
		& \qquad \qquad \qquad  \qquad \qquad \ \ \  -(f_1^{0}(x_1)/g(x_2,\ldots,x_n))\sin(10\pi f_1)\nonumber
	\end{align}
	where $x_i\in[0,1]$, $i=1,\ldots,n$, and $n=2,\ldots,30$. The efficient frontier is formed with $g(x_2,\ldots,x_n)=1$.
	\item Fonseca and Fleming (FF) function has a non-convex efficient frontier:
	\begin{align}\label{func4}
		&\min\  f^0_1(x), f^0_2(x) \nonumber \\
		&\text{where\ }f^0_1(x)= 1- \exp\left(-\sum_{i=1}^n \left(x_i-\frac{1}{\sqrt{n}}\right)^2\right)  \nonumber \\
		&\text{and\ } f^0_2(x)= 1- \exp\left(-\sum_{i=1}^n \left(x_i+\frac{1}{\sqrt{n}}\right)^2\right)\nonumber
	\end{align} 
	where  $-4\leq x_i \leq 4, i=1,\ldots,n$ and $n=2,\ldots,30$.
\end{itemize}

\end{document}